\documentclass[11 pt]{amsart}
\usepackage{lineno,hyperref}
\usepackage{amsmath,amssymb,amsfonts, enumerate, mathtools, tabularx, bigstrut, setspace, multirow, float}

\DeclareMathOperator{\SL}{\operatorname{SL}}
\DeclareMathOperator{\ad}{\operatorname{ad}}
\DeclareMathOperator{\Ad}{\operatorname{Ad}}
\DeclareMathOperator{\Tr}{\operatorname{Tr}}
\DeclareMathOperator{\id}{\operatorname{id}}

\newcommand{\Fig}[1]{\parbox{12pt}{\Huge #1}}
\def\Zero{\Fig{0}}

\newtheorem{theorem}{Theorem}[section]
\newtheorem*{theorem*}{Theorem}
\newtheorem{corollary}[theorem]{Corollary}

\newtheorem{lemma}[theorem]{Lemma}
\theoremstyle{definition}
\newtheorem{definition}[theorem]{Definition}
\newtheorem{example}[theorem]{Example}
\newtheorem{remark}[theorem]{Remark}

\begin{document}

\title{On the Positivity of Kirillov's Character Formula
}

\author{Ehssan Khanmohammadi}

\email{ehssan@pm.me}

\maketitle

\begin{abstract}
We give a direct proof for the positivity of Kirillov's character on the convolution algebra of smooth, compactly supported functions on a connected, simply connected nilpotent Lie group $G$. Then we use this positivity result to construct a representation of $G\times G$ and establish a $G\times G$-equivariant isometric isomorphism between our representation and the Hilbert--Schmidt operators on the underlying representation of $G$. In fact, we provide a more general framework in which we establish the positivity of Kirillov's character for coadjoint orbits of groups such as $\SL(2, \mathbb{R})$ under additional hypotheses that are automatically satisfied in the nilpotent case. These hypotheses include the existence of a real polarization and the Pukanzsky condition.
\end{abstract}

\section{Introduction}
In his fundamental paper \cite{K62}, Kirillov proved that coadjoint orbits of a connected, simply connected nilpotent Lie group correspond, under quantization, to the equivalence classes of its irreducible unitary representations. The theory of geometric quantization due to Kirillov, Kostant~\cite{K70}, and Souriau~\cite{S70} has shown that this close connection extends to many other groups. 

Kirillov's character formula says that the characters of irreducible unitary representations of a Lie group $G$ ``should'' be given by an equation of the form
	\begin{equation}\label{E:CharacterFormula}
	\Tr\pi_{\mathcal{O}}(\exp X)
	=
	j^{-1/2}(X)\int_{\mathcal{O}}e^{i\ell(X)+\sigma}
	\end{equation}
where $\mathcal{O}$ is the coadjoint orbit in $\mathfrak{g}^*$ corresponding to $\pi_{\mathcal{O}}\in \widehat{G}$, $\sigma$ is a canonical symplectic form on $\mathcal{O}$, and $j$ is the analytic function on $\mathfrak{g}$ defined by the formula
	\[
	j(X)=\det\left(\frac{\sinh(\ad X/2)}{\ad X/2}\right).
	\]
Kirillov proved his character formula for simply connected nilpotent and simply connected compact Lie groups \cite{K62,K68} and conjectured its universality. The validity of this conjecture has been verified for some other classes of Lie groups, most notably for the case of tempered representations of reductive Lie groups by Rossmann~\cite{R78}; see also \cite{V79}. Moreover,  Atiyah--Bott~\cite{AB84} and Berline--Vergne~\cite{BV82} following the work of Duistermaat--Heckman~\cite{DH82}, have shown that for compact Lie groups, Kirillov's character formula is equivalent to the Weyl character formula.

Our main contributions in this paper are as follows. Rather than construct representations, compute their characters, and compare the results with
Kirillov's character formula, we work directly with Kirillov's formula and give direct arguments to
prove positivity in the sense of operator algebras. More precisely, we give a direct proof for the positivity of Kirillov's character on the convolution algebra of the Schwartz functions on a connected nilpotent Lie group. The fact that Kirillov's character defines a positive trace is remarkable in light of the Gelfand--Naimark--Segal (GNS) construction from operator algebra theory, since GNS implies, roughly speaking, that any positive linear functional on a Lie group is the character of a group representation. In the last section, we use our positivity result and the GNS to construct a representation of $G\times G$ for a connected, simply connected nilpotent Lie group $G$. Then in Theorem~\ref{T:comparison of quantizations} we discuss the relation between the representations that we have obtained from the GNS construction, and the Kirillov representations corresponding to the coadjoint orbits of $G$. Indeed, we prove that there is a $G\times G$-equivariant isometric isomorphism between our representation and the Hilbert--Schmidt operators on the underlying representation of $G$. We deal with nilpotent groups this way, thereby recovering some of the first results obtained by Kirillov. Moreover, in Theorem~\ref{T:Positivity for SL(2, R)}, we prove the positivity of Kirillov's character for coadjoint orbits of unimodular groups such as $\SL(2, \mathbb{R})$ in the presence of supplementary geometric hypotheses on the orbits that are automatically satisfied in the nilpotent case. One hypothesis is the existence of a real polarization; a second is a completeness hypothesis on the
(locally affine) fibers of the associated Lagrangian fibration.
\section{Some Notations and Background Material}\label{C:Background}
In this section we fix some notations and collect several standard results that we will use freely in the sequel. 
\begin{definition}\label{D:modular-function}
Let $G$ be a locally compact group $G$ with a left Haar measure $d\mu$. The \emph{modular function} $\Delta\colon G\to (0, \infty)$ is defined via
	\[
	\Delta(y) \int_G f(xy)\, d\mu(x)
	=
	\int_G f(x)\, d\mu(x)
	\]
for $x$ and $y$ in $G$ and all $f\in L^1(G)$ so that
	\[
	\int_G f(x^{-1})\, d\mu(x)
	=
	\int_G f(x)\Delta(x^{-1})\, d\mu(x).
	\]
\end{definition}
If $G$ is a Lie group, then it is well known that $\Delta(g)=|\det \Ad(g^{-1})|$. This implies, for instance, that any connected nilpotent Lie group is unimodular, that is, for such groups the modular function is constant and everywhere equal to one.
\begin{definition} Let $G$ be a locally compact group with a left Haar measure $d\mu$. Given functions $f$ and  $g$ on $G$, their \emph{convolution} $f*g$ is the function on $G$ defined by
	\begin{align*}
	f*g(x)
	&=
	\int_Gf(y)g(y^{-1}x)\, d\mu(y)
	\\
	&=
	\int_Gf(xy)g(y^{-1})\, d\mu(y)
	\end{align*}
whenever one (and hence both) of these integrals makes sense.
\end{definition}
Let $(\pi, V_\pi)$ be a unitary representation of a locally compact group $G$. Then $\pi$ induces a continuous homomorphism of Banach $*$-algebras from $L^1(G)$ to the space of bounded operators $\mathcal{B}(V_{\pi})$ via Bochner integration. By abuse of notation, we denote this induced $*$-homomorphism by $\pi$. Hence,
	\[
	\pi\colon L^1(G)\to \mathcal{B}(V_\pi) \quad \text{and} \quad \pi(f)=\int_G f(g)\pi(g)\, dg.
	\]
Here we are using the facts that $L^1(G)$ is a Banach algebra under convolution, which is just the usual group ring $\mathbb{C}G$ if $G$ is finite, and that $L^1(G)$ has a natural isometric (conjugate-linear) involution $f\mapsto f^*$, where 
	\[
	f^*(x)=\overline{f(x^{-1})}\Delta (x^{-1}).
	\]
Note that these definitions are formulated so that $f*g=L(f)(g)$ for the left regular representation $(L, L^1(G))$.
\begin{definition} We say that an irreducible unitary representation $\pi$ of a Lie group $G$ has a \emph{global or Harish-Chandra character} $\Theta$ if $\pi(f)$ is of trace class for all $f\in C_c^\infty(G)$ and moreover $f\mapsto \Theta(f)=\Tr\pi(f)$ is a distribution.
\end{definition}
The next two results provide sufficient conditions for  $\pi(f)$ to be trace class.
\begin{theorem}[Harish-Chandra \cite{Harish56}] Let $G$ be a connected semisimple Lie group with finite center. Then for every irreducible unitary representation $\pi$ of $G$ and every  $f\in C_c^{\infty}(G)$, the operator $\pi(f)$ is trace class and the map $f\mapsto \Tr \pi(f)$ is a distribution on $G$.
\end{theorem}

\begin{theorem}[Kirillov \cite{K62}]\label{T:trace class-nilpotent} Let $G$ be a nilpotent Lie group. Then for every irreducible unitary representation $\pi$ of $G$ and every Schwartz function $f\in \mathcal{S}(G)$, the operator $f\mapsto \Tr \pi(f)$ is trace class and the map $f\mapsto \Tr \pi(f)$ is a tempered distribution on $G$.
\end{theorem}
Let $(\pi, V_{\pi})$ be an irreducible unitary representation. Observe that the above discussions imply that $\pi(f*f^*)=\pi(f)\pi(f)^*$. Thus, the distributional character of $(\pi, V_{\pi})$, whenever defined, is \emph{convolution-positive} (or \emph{positive} for short) in the sense that
	\begin{equation}\label{E:positivity of trace}
	\Tr \pi(f*f^*)=\Tr\pi(f)\pi(f)^*=\|\pi(f)\|_{\text{HS}}^2\ge 0.
	\end{equation}
To close this section, we introduce Kirillov's character formula which says that the characters $\chi$ of irreducible unitary representations of a Lie group $G$ ``should'' be given by an equation of the form
	\begin{equation}
	\chi(\exp X)=j^{-1/2}(X)\int_{\mathcal{O}} e^{i\ell(X)} \, d\mu_{\mathcal{O}}(\ell).
	\end{equation}
Here $\mathcal{O}$ is a coadjoint orbit in $\mathfrak{g}^*$, $\mu_\mathcal{O}$ is the canonical (or Liouville) symplectic measure on $\mathcal{O}$, and $j$ is the Jacobian of the exponential map $\exp\colon \mathfrak{g}\to G$ given by the formula
	\[
	j(X)=\det\left(\frac{\sinh(\ad X/2)}{\ad X/2}\right)
	\]
whenever $G$ is unimodular. This character formula should be interpreted as an equation of distributions on a certain space of test functions on $\mathfrak{g}$ as follows. For all smooth functions $f$ compactly supported in a sufficiently small neighborhood of the origin in $\mathfrak{g}$,
	\begin{equation}\label{E:CharacterFormulaInt}
		\Tr \int_{\mathfrak{g}} f(X)\pi(\exp X)\, dX=
	\int_{\mathcal{O}}\int_{\mathfrak{g}}
	 e^{i\ell(X)} f(X)j^{-1/2}(X)\, dX d\mu_{\mathcal{O}}(\ell)
	\end{equation}
where $\pi$ is the representation with character $\chi$.
\begin{theorem}[Kirillov \cite{K62}]\label{T:CharacterFormula} Suppose $G$ is a connected nilpotent Lie group. Then the irreducible unitary representations of $G$ are in natural one-to-one correspondence with the integral orbits of $G$.
Moreover, if $G$ is simply connected and $\pi$ is an irreducible unitary representation of $G$ with the associated coadjoint orbit $\mathcal{O}$, then \eqref{E:CharacterFormulaInt} holds.
\end{theorem}
\section{Nilpotent Lie Groups}\label{S:Nilpotent Lie Groups}
Nilpotent Lie groups and their representations have been studied extensively in the literature. From the many papers by Corwin, Greanleaf, Lipsman, Pukanzsky and others we only cite \cite{KL72}, \cite{P67}, and \cite{S68} and refer the reader to~\cite{CG90} for a more comprehensive list of bibliographies. For a nilpotent Lie group $G$, let $\chi_\mathcal{O}$ denote Kirillov's character corresponding to a coadjoint orbit $\mathcal{O}$ of $G$. That is,
	\[
	\chi_\mathcal{O}(f)=\int_{\mathcal O}\widehat{f\circ \exp}(\ell)\, d\mu_\mathcal{O}(\ell)=\int_{\mathcal O}\int_\mathfrak{g} f(\exp X) e^{i\ell(X)}\, dX \, d\mu_\mathcal{O}(\ell),\quad f\in \mathcal{S}(G).
	\]
We would like to show directly that if $G$ is a connected nilpotent Lie group, then for all $f\in \mathcal{S}(G)$,
	\[
	\chi_\mathcal{O}(f*f^*)	\ge 0. 
	\]
We first prove this result for the important case of the Heisenberg group where computations are short and insightful.

The Heisenberg group $H$ can be realized by the $3\times 3$ upper triangular matrices
	\[
	H=
	\left\{
	\begin{bmatrix}
	1&x&z\\
	0&1&y\\
	0&0&1
	\end{bmatrix}
	\Biggm |
	x, y, z\in \mathbb{R}
	\right\}
	\]
with the usual matrix multiplication. 
The Lie algebra $\mathfrak{h}$ of the Heisenberg group has three generators $X$, $Y$, and $Z$ satisfying the commutation relation $[X, Y]=Z$.
\begin{example}[Positivity of Kirillov's Character for $H$]
Choose $\ell \in \mathfrak{h}^*$ with $\ell(Z)=\gamma\ne 0$ so that the coadjoint orbit $\mathcal{O}$ through $\ell$ is the plane $Z^*=\gamma$ in the $X^*Y^*Z^*$-coordinate system in $\mathfrak{h}^*$. Let $f\in \mathcal{S}(H)$ be a function in the Schwartz space of the Heisenberg group and identify the coadjoint orbit $\mathcal{O}$ and the Lie algebra $\mathfrak{h}$ with the plane 
$\mathbb{R}^2\times \{\gamma\}=\{(\alpha, \beta, \gamma)\mid \alpha, \beta\in \mathbb{R}\}$ 
and 
$\mathbb{R}^3=\{(a, b, c)\mid a, b, c\in \mathbb{R}\}$, respectively. The Liouville measure on $\mathcal{O}$ is $\mu_\mathcal{O}=\frac{1}{2\pi\gamma}dX\wedge dY$. We compute
	\begin{align*}
	\chi_{\mathcal{O}}(f)&=
	\int_{\mathcal{O}}\int_{\mathfrak{h}}
		f(\exp X) e^{i\phi(X)}\,dX\,d\mu_{\mathcal{O}}(\phi)
		\\
	&=\frac{1}{2\pi|\gamma|}\int_{\mathcal{O}}\int_{\mathfrak{h}}
		f(\exp (a, b, c)) e^{i(a\alpha+b\beta+c\gamma)}\, 	da\,db\,dc\,d\alpha\, d\beta
		\\
	&=
	\frac{1}{2\pi|\gamma|}
	\int_c\int_{(\alpha, \beta)}\int_{(a, b)}
		f(\exp (a, b, c))e^{i(a\alpha+b\beta+c\gamma)}
	\, da\,db\, d\alpha\,d\beta\, dc
	\\
	&=
	\frac{2\pi}{|\gamma|}
	\int_c 
		f(\exp(0, 0, c))e^{ic\gamma}\, dc
	\end{align*}
where in the last step we have used the Fourier inversion theorem. Thus, we have found the following simplified form of Kirillov's formula for the coadjoint orbit $\mathcal{O}$ of the Heisenberg group:
	\[
	\chi_{\mathcal{O}}(f)=
	\frac{2\pi}{|\gamma|}
	\int_{\mathbb{R}}
		f(\exp tZ)e^{it\gamma}\, dt, \quad f\in \mathcal{S}(H).
	\]
Now we use this to check the positivity of $\chi_\mathcal{O}$. First note that
	\[
	\chi_{\mathcal{O}}(f*f^*)=
	\frac{2\pi}{|\gamma|}
	\int_{\mathbb{R}}\int_{H}
	f(\exp tZ\cdot h)\overline{f(h)} e^{it\gamma}\, d\lambda(h)\, dt,
	\]	
where $d\lambda=dx\wedge dy\wedge dz$ is the Haar measure on $H$, and that we can set $h=\exp(xX+yY+zZ)$ due to the surjectivity of $\exp\colon \mathfrak{h}\to H$. Since $\exp Z$ is in the center $Z(H)$ of $H$ we may rearrange the last expression to find
	
	\begin{align*}
	&\frac{2\pi}{|\gamma|}
	\int_{\mathbb{R}}\int_{\mathbb{R}^3}
		f(\exp(xX+yY+(z+t)Z))\overline{f(\exp(xX+yY+zZ))} 			e^{it\gamma}\, d\lambda\, dt
	\\
	&=
	\frac{2\pi}{|\gamma|}
	\int_{\mathbb{R}}\int_{\mathbb{R}^3}
		f(\exp(xX+yY+(z+t)Z))e^{i(z+t)\gamma}
		\overline{f(\exp(xX+yY+zZ))e^{iz\gamma}} \, d\lambda\, 		dt
	\\
	&=
	\frac{2\pi}{|\gamma|}
	\int_{\mathbb{R}^2}\left |\int_{\mathbb{R}}
		f(\exp(xX+yY+zZ))e^{iz\gamma} \, dz\right |^2 \, dy\, dx\ge 0,
	\end{align*}
as desired. Kirillov's character is also positive for any point-orbit $\mathcal{O}=\{\ell\}$ in the $X^*Y^*$-plane:
\begin{align*}
	\chi_{\mathcal{O}}(f*f^*)
	=
	\int_{\mathfrak{h}} 
		f*f^*(\exp X)e^{i\ell (X)}
		\, dX
	&=
	\int_H 
		f*f^*(x)e^{i\ell(\log x)}
	\, d\lambda(x)
	\\
	&=
	\int_H
		\int_H
			f(xy)\overline{f(y)}
		\, d\lambda(y)
		e^{i\ell(\log x)}
	\, d\lambda(x)
	\\
	&=
	\int_H\int_H
			f(x)\overline{f(y)}
			e^{i\ell(\log xy^{-1})}
	\, d\lambda(x)\, d\lambda(y)
	\\
	&=
	\left|
		\int_H f(x)e^{i\ell(\log x)}
		\, d\lambda(x)
	\right|^2\ge 0.
\end{align*}
The last equality is due to the facts that if $\xi, \eta\in \mathfrak{h}$, then 
	\[
	\exp\xi\exp \eta
	=
	\exp(\xi+\eta+\frac{1}{2}[\xi, \eta])
	\]
and that since $[\xi, \eta]\in \mathfrak{z}(\mathfrak{h})$, $\ell([\xi, \eta])=0$. Therefore, the map from $H$ to $\mathbb{C}$ defined by $x\mapsto \ell(\log x)$ is a group homomorphism.
\end{example}
To proceed with the case of general nilpotent Lie groups and beyond, we recall two definitions.
\begin{definition}\label{D:real-polarization} Let $\mathfrak{g}$ be a real finite-dimensional Lie algebra and let $\ell\in \mathfrak{g}^*$. A subalgebra $\mathfrak{m}\subset \mathfrak{g}$ is said to be a \emph{real polarization} subordinate to $\ell$ provided that $\ell([X, \mathfrak{m}])=0$ if and only if $X\in \mathfrak{m}$. 
\end{definition}
If $\mathfrak{g}$ is a nilpotent Lie algebra, and $\ell\in \mathfrak{g}^*$, then there exists a polarizing subalgebra subordinate to $\ell$. See \cite[Theorem 1.3.3]{CG90}.
\begin{definition} Let $G$ be a Lie group with Lie algebra $\mathfrak{g}$. Then the \emph{stabilizer subgroup} associated with $\ell\in \mathfrak{g}^*$ is
	\[
	R_\ell=\{g\in G\mid \Ad^*(g)\ell=\ell\}
	\]
which has the corresponding \emph{radical Lie algebra}
	\[
	\mathfrak{r}_\ell=\{X\in \mathfrak{g}\mid \ad^*(X)\ell=0\}.
	\]
\end{definition}
Throughout, we shall use the notations introduced in these definitions without further comment. 
Any polarizing subalgebra subordinate to $\ell$ lies halfway between the radical $\mathfrak{r}_\ell$ and the Lie algebra $\mathfrak{g}$. This motivates the following lemma.
\begin{lemma}\label{L:Orbit Diffeomorphism} Let $G$ be a Lie group with Lie algebra $\mathfrak{g}$ and a polarizing subalgebra $\mathfrak{m}$ subordinate to an element $\ell\in \mathfrak{g}^*$. Let $M=\exp \mathfrak{m}$. Then the mapping 
	\begin{align*}
	\theta\colon M/R_\ell&\to (\mathfrak{g}/\mathfrak{m})^*\\
	mR_\ell &\mapsto  m\cdot \ell -\ell.
	\end{align*}
is a diffeomorphism, provided that the orbit of $\ell$ under $M$, $\Ad^*(M)\ell$, is closed in $\mathfrak{g}^*$. In particular, $\theta$ is automatically a diffeomorphism if $G$ is a connected nilpotent Lie group.
\end{lemma}
\begin{proof} Note that $\theta$ is well defined, because by the identity $\Ad^*(\exp X)= e^{\ad^* X}$ we have $m\cdot \ell (P)=\ell(P)$ for $P\in \mathfrak{m}$ and $m\in M$. Injectivity is clear since $ m_1\cdot\ell-\ell=m_2\cdot \ell-\ell$ gives $m_1^{-1}m_2\in R_\ell$ and hence $m_1R_\ell=m_2 R_\ell$. Surjectivity of $\theta$, on the other hand, is equivalent to surjectivity of its lift $\tilde{\theta}$ to the group $M$ defined by $m\mapsto m\cdot \ell -\ell$. Computing the differential of the equivariant map $\tilde{\theta}$ at a point $g\in M$ one infers that 
	\[
	\tilde{\theta}_{*, g}(X)=\frac{d}{dt}\Bigr|_{t=0}\tilde{\theta}(g\exp tX )=g\cdot((\ad^* X) \ell) \quad \text{ for } X\in T_gM,
	\]
and 
	\begin{align*}
	\text{rank } \tilde{\theta}_{*, g}&=\dim T_gM-\dim \{X\in T_gM\mid (\ad^* X) \ell=0\}\\
	&=\dim \mathfrak{m}-\dim \mathfrak{r}_\ell.
	\end{align*}
Consequently, $\tilde{\theta}$ is a submersion, indeed a local diffeomorphism, and hence an open map. Because $\Ad^*(M)\ell$ is closed in $\mathfrak{g}^*$ by assumption, we conclude that the image of $\tilde{\theta}$ is a closed submanifold of $(\mathfrak{g}/\mathfrak{m})^*$ and the surjectivity of $\theta$ follows. The last assertion of the theorem is a consequence of a result of Chevalley and Rosenlicht \cite[Theorem 3.1.4]{CG90} which states that if $G$ acts unipotently on a real vector space $V$, then the $G$-orbits are closed in $V$.
\end{proof}
Variations of the closedness assumption used in the statement of this lemma are known as the \emph{Pukanzsky condition} in the literature.
\begin{theorem}[Weil's formula]\label{T:QuotientIntegral}
Let $G$ be a locally compact group, and let $H$ be a closed subgroup. There exists a $G$-invariant Radon measure $\nu\ne 0$ on the quotient $G/H$ if and only if the modular functions $\Delta_G$ and $\Delta_H$ agree on $H$. In this case, the measure $\nu$ is unique up to a positive scalar. Given Haar measures on $G$ and $H$, there is a unique choice for $\nu$ such that for every $f\in C_c(G)$ one has the quotient integral formula
	\[
	\int_G f(g)\, dg=\int_{G/H}\int_H f(xh)\, dh\, d\nu(xH).
	\]
\end{theorem}
To lighten the notation we shall write $dxH$ for $d\nu(xH)$. 

The following theorem is proved by Lipsman~\cite{Lips}. Since we are mainly concerned with establishing positivity results, we shall not go into any discussion of the normalizations of measures  and instead refer the reader to \cite{EK15} for the details.
\begin{theorem}\label{T:Lipsman-Pukanzsky} Let $G$ be a connected nilpotent Lie group, $\mathfrak{m}$ a polarizing subalgebra subordinate to some $\ell\in \mathfrak{g}^*$, and $M=\exp \mathfrak{m}$. Then for any $f\in \mathcal{S}(G)$,
	\begin{equation}\label{E:L-P Quotient Integral}
	\chi_\mathcal{O}(f)=\int_{G/M}\int_M f(xpx^{-1})e^{i\ell( \log p)}\,dp\,dxM,
	\end{equation}
where $\chi_\mathcal{O}$ is the Kirillov character for the coadjoint orbit $\mathcal{O}$ through $\ell$.
\end{theorem}
If $G$ is a connected, simply connected nilpotent Lie group, Kirillov's character is known to be positive for any coadjoint orbit of $G$ in view of Theorems~\ref{T:trace class-nilpotent} and \ref{T:CharacterFormula}, and Equation~\eqref{E:positivity of trace}. One may prove this positivity result directly and without giving any reference to the underlying representation by applying Theorem~\ref{T:Lipsman-Pukanzsky}. See \cite{EK15} for this approach. In the next section we develop a more general framework to treat the positivity statements uniformly, and we obtain our desired result for nilpotent Lie groups in Corollary~\ref{C:positivity-of-nilpotents-and-sl2}.
\section{$\SL(2,\mathbb{R})$ and Beyond}\label{S: SL2}
Rossmann has shown that Kirillov's formula is valid for characters of irreducible tempered representations of semisimple Lie groups. (Characters of non-tempered irreducible representations, on the other hand, usually do not arise as Fourier transforms of invariant measures on coadjoint orbits.)

In this section we prove the positivity of Kirillov's character for coadjoint orbits of unimodular Lie groups such as $\SL(2, \mathbb{R})$ under a few conditions that are automatically satisfied in the nilpotent case. One hypothesis is the geometric condition required by Lemma~\ref{L:Orbit Diffeomorphism}; the other is the existence of real polarizations satisfying Equation~\eqref{E:Relation between j-functions-hyperbolic} that we shall first encounter while studying $\SL(2, \mathbb{R})$.
	
To carry out our computations, we need to be able to decompose the Haar measure of a locally compact group $G$ over a closed subgroup $H$ and the quotient $G/H$ when $G/H$ does not necessarily admit a $G$-invariant measure. 
\begin{definition}\label{D:Quasi-invariance} Let $G$ be a locally compact group and $H$ a closed subgroup of $G$. A Radon measure $\mu$ on $G/H$ is called \emph{quasi-invariant} under $G$ if there exist functions $\lambda_g$ defined on $G/H$ such that for all $g\in G$ and $f\in C_c(G/H)$
	\[
	\int_{G/H} f(\Lambda_g x)\, d\mu(xH)=\int_{G/H}f(x)\lambda_g(x)\, d\mu(xH),
	\]
where $\Lambda_g(xH)=g^{-1}xH$. 
\end{definition}
Note that if $\lambda_g=1$ for all $g\in G$, then $\mu$ is invariant under $G$, and hence quasi-invariance extends the notion of invariance.

The next theorem generalizes Weil's formula, Theorem~\ref{T:QuotientIntegral}.
\begin{theorem}[Mackey--Bruhat]\label{T:RhoFunction}
Let $G$ be a locally compact group. Given a closed subgroup $H$ of $G$, there is always a continuous, strictly positive solution $\rho$ of the functional equation
	\begin{equation}\label{E:Functional equation}
	\rho(xh)
	=
	\rho(x)\frac{\Delta_H(h)}{\Delta_G(h)},
	\quad x\in G, h\in H.
	\end{equation}
Moreover, there is a quasi-invariant measure $d_\rho xH$ on $G/H$ such that 
	\[
	\int_G f(g)\rho(g)\, dg=\int_{G/H}\int_H f(xh)\, dh\, d_\rho xH.
	\]
\end{theorem}

\begin{remark} Henceforth we will assume that our quotient spaces are equipped with measures as in Theorem~\ref{T:RhoFunction} and we will drop the index $\rho$ in the measure. One can show that in this situation
	\[
	\lambda_g(xH)=\frac{\rho(gx)}{\rho(x)}
	\quad \text{for}\quad x, g\in G.
	\]
See Definition~\ref{D:Quasi-invariance} and, for instance, \cite[Proposition 8.1.4]{RS00}. If $G/H$ carries a $G$-invariant measure, then we shall assume that $\rho\equiv 1$; this happens, for instance, when we study quotients diffeomorphic to coadjoint orbits which naturally carry invariant symplectic measures.
\end{remark}
Now we record a corollary to Theorem~\ref{T:RhoFunction}.
\begin{corollary}\label{Co:Quotient integral formula for quasi-invariant measures}
Let $G$ be a locally compact group with closed subgroups $H$ and $K$ such that $K\subset H\subset G$. Then there exist suitably normalized quasi-invariant measures on the quotient spaces such that the equality
	\[
	\int_{G/K} f(g)\,dgK=
	\int_{G/H}\int_{H/K}f(xh)\frac{\rho_{(G, K)}(xh)}{\rho_{(G, H)}(xh)\rho_{(H, K)}(h)}\, dhK\, dxH
	\]
holds for any $f\in C_c(G/K)$.	
\end{corollary}

\begin{proof}
Let $F\in C_c(G)$ and apply Theorem~\ref{T:RhoFunction} twice to obtain
	\begin{align*}
	\int_G F(t)\, dt&=
	\int_{G/K}\int_K \frac{F(gk)}{\rho_{(G, K)}(gk)} \, dk\, dgK\\
	&=\int_{G/H}\int_H \frac{F(xs)}{\rho_{(G, H)}(xs)} \, ds\, dxH\\
	&=\int_{G/H}\left(\int_{H/K}\int_K \frac{F(xhk)}{\rho_{(G, H)}(xhk)\rho_{(H, K)}(hk)}\, dk\, dhK\right)dxH.
	\end{align*}
Fix $f\in C_c(G/K)$ and choose a function $\alpha\in C_c(G)$ with the property that for every $gK\in \text{supp } (f)$ we have  $\int_K \alpha (gk)\, dk =1$, then substitute $F=f\alpha\rho_{(G, K)}$. For the standard proof of existence of such $\alpha$ we refer to~\cite[Lemma 2.47]{F95}. The above computation implies
	\begin{align*}
	\int_{G/K} 
		&
		f(g)
	\, dgK
	=
	\int_{G/H}
	\int_{H/K}
	\int_K 
		\frac{f(xhk)\alpha(xhk)\rho_{(G, K)}(xhk)}	{\rho_{(G, H)}(xhk)\rho_{(H, K)}(hk)}
		\, dk\, dhK\,dxH
	\\
	&=
	\int_{G/H}\int_{H/K}
		f(xh)
		\int_K 
			\frac{
				\alpha(xhk)\rho_{(G, K)}(xh)
				\frac{\Delta_K(k)}{\Delta_G(k)}
			}
			{
			\rho_{(G, H)}(xh)
			\frac{\Delta_H(k)}{\Delta_G(k)}
			\rho_{(H, K)}(h)
			\frac{\Delta_K(k)}{\Delta_H(k)}
			}
	\, dk\, dhK\,dxH
	\\
	&=
	\int_{G/H}\int_{H/K}
		f(xh)
		\frac{
		\rho_{(G, K)}(xh)}
		{\rho_{(G, H)}(xh)\rho_{(H, K)}(h)
		}
	\, dhK\, dxH.
	\end{align*}
\end{proof}
The next lemma will be used in conjunction with Corollary~\ref{Co:Quotient integral formula for quasi-invariant measures}.
\begin{lemma}\label{L:rho-products}\hfill
\begin{enumerate}[(i)]
\item Let $G$ be a locally compact group, and let $H$ be a unimodular subgroup of $G$. Then the function $\rho_{(G, H)}\colon G\to (0, \infty)$ defined by $\rho_{(G, H)}=\Delta_G^{-1}$ is a solution to the functional equation~\eqref{E:Functional equation}. \label{L:unimodular-subgroup-rho}
\item Let $G$ be a unimodular Lie group with Lie algebra $\mathfrak{g}$. Suppose $\ell\in \mathfrak{g}^*$, and let $M$ and $R_\ell$ denote a polarizing subgroup and the stabilizer subgroup associated with $\ell$, respectively. Then one can let $\rho_{(M, R_\ell)}=\Delta^{-1}_M$ in which case any $\rho$-function $\rho_{(G, M)}$ satisfies the identity
	\[
	\rho_{(G, M)}(xm)\rho_{(M, R_\ell)}(m)=\rho_{(G, M)}(x)
	\]
for all $x\in G$ and $m\in M$.
\end{enumerate}
\end{lemma}

\begin{proof}
The first part is a consequence of the fact that the modular function is a group homomorphism. As for the second part, note that since $G$ is unimodular and $\mathcal{O}\cong G/R_\ell$ carries a $G$-invariant measure, the stabilizer $R_\ell$ is also unimodular by Theorem~\ref{T:QuotientIntegral}. Thus, Part~\eqref{L:unimodular-subgroup-rho} implies that $\rho_{(M, R_\ell)}\colon M\to (0, \infty)$ given by $\rho_{(M, R_\ell)}=\Delta^{-1}_M$ is a solution to the functional equation~\eqref{E:Functional equation}. But for any $\rho$-function $\rho_{(G, M)}\colon G\to (0, \infty)$ satisfying \eqref{E:Functional equation} we have 
\begin{equation}\label{E:Functional Equation (G, M)}
\begin{aligned}
\rho_{(G, M)}(xm)&=\rho_{(G, M)}(x)\frac{\Delta_M(m)}{\Delta_G(x)}\\
&=\rho_{(G, M)}(x)\Delta_M(m),\text{ for } x\in G, \text{ and } 
	m
	\in M.
\end{aligned}
\end{equation}
This combined with our choice of $\rho_{(M, R_\ell)}=\Delta^{-1}_M$ finishes the proof of the second part.
\end{proof}

For a Lie group $G$ with Lie algebra $\mathfrak{g}$ write $j_{\mathfrak{g}}$ for the Jacobian of the exponential map $\exp\colon \mathfrak{g}\to G$, with reference to the Lebesgue measure on $\mathfrak{g}$ and the left Haar measure on $G$. According to a classical result of F.~Schur, 
	\[
	j_{\mathfrak{g}}(X)
	=
	\det\left(
	\frac{\id-e^{-\ad X}}{\ad X}
	\right),\quad 	
	X\in \mathfrak{g}.
	\]
In the next example we compute the $j$-function for $\mathfrak{g}=\mathfrak{sl}(2, \mathbb{R})$ and one of its subalgebras and point out a relation between the two.

\begin{example}\label{Ex:comparison of j-functions} To streamline the notation, we write
	\begin{equation}\label{E:A basis for sl(2, R)}
	H=
	\begin{bmatrix}
	1&0\\
	0&-1
	\end{bmatrix},
	X=
	\begin{bmatrix}
	0&1\\
	1&0
	\end{bmatrix}, \text{ and }
	Y=
	\begin{bmatrix}
	0&1\\
	-1&0
	\end{bmatrix}
	\end{equation}
	for a basis of the three dimensional Lie algebra $\mathfrak{g}=\mathfrak{sl}(2, \mathbb{R})$ of $2\times 2$ traceless matrices. Consider the polarizing subalgebra
	\[
	\mathfrak{m}=\operatorname{Span}\{X+Y, H\}
	\]
subordinate to $\ell=H^*$ consisting of upper triangular matrices in $\mathfrak{g}$. 
For $W=a H+ b(X+Y)$, the matrix of $\ad_{\mathfrak{m}} W\colon \mathfrak{m}\to \mathfrak{m}$ 
with respect to the basis 
$\mathfrak{B}'=\{ X+Y, H\}$
of $\mathfrak{m}$ is given by
	\[
	[\ad_{\mathfrak{m}} W]_{\mathfrak{B}'}
	=
	\begin{bmatrix*}[r]
	2a&-2b\\0&0
	\end{bmatrix*}.
	\]
Therefore,
	\begin{equation}\label{E:j-sub-m}
	j_{\mathfrak{m}}(W)
	=
	\det\left(
	\frac{\id-e^{-\ad_{\mathfrak{m}} W}}{\ad_{\mathfrak{m}} W}
	\right)
	=
	\frac{1-e^{-2a}}{2a}
	\end{equation}
by the spectral mapping theorem applied to the eigenvalues of $\ad_{\mathfrak{m}} W$, namely $2a$ and $0$. Extend the basis $\mathfrak{B}'$ of $\mathfrak{m}$ to the basis 
$\mathfrak{B}=\{X+Y, H, X\}$
of $\mathfrak{g}$. Then the matrix of $\ad_{\mathfrak{g}} W\colon \mathfrak{g}\to \mathfrak{g}$ 
with respect to the basis $\mathfrak{B}$ is
	\[
	[
	\ad_{\mathfrak{g}} W
	]_{\mathfrak{B}'}
	=
	\begin{bmatrix*}[r]
	2a&-2b&2a\\
	0&0&2b\\
	0&0&-2a
	\end{bmatrix*}.
	\]
Therefore,
	\begin{equation}\label{E:j-sub-g}
	j_{\mathfrak{g}}(W)
	=
	\det\left(
	\frac{\id-e^{-\ad_{\mathfrak{g}} W}}{\ad_{\mathfrak{g}} W}
	\right)
	=
	\left(
	\frac{1-e^{-2a}}{2a}
	\right)
	\left(
	\frac{1-e^{2a}}{-2a}
	\right)
	=
	\frac{(e^a-e^{-a})^2}{4a^2}
	\end{equation}
by the spectral mapping theorem applied to the eigenvalues of $\ad_{\mathfrak{g}} W$, namely $2a$, $0$, and $-2a$. Equations~\eqref{E:j-sub-m}, and \eqref{E:j-sub-g} reveal an interesting relation between $j_{\mathfrak{m}}(W)$ and $j_{\mathfrak{g}}(W)$, namely
	\begin{equation}\label{E:Relation between j-functions}
	j_{\mathfrak{g}}(W)
	=
	j_{\mathfrak{m}}^2(W)/
	\Delta_M(\exp W).
	\end{equation}
Using the fact that 
$\Delta_M(\exp W)=\det \Ad_M(\exp -W)=
\det e^{-\ad_{\mathfrak{m}} W}$ we obtain
	\begin{align*}
	j_{\mathfrak{m}}^2(W)/
	\Delta_M(\exp W)
	&=
	\det\left(
	\frac{\id-e^{-\ad_{\mathfrak{m}} W}}{\ad_{\mathfrak{m}} W}
	\right)^2
	\det e^{\ad_{\mathfrak{m}} W}
	\\
	&=
	\det
	\left(
	\frac{\sinh(\ad_{\mathfrak{m}} W/2)}
		{\ad_{\mathfrak{m}} W/2}
	\right)^2.
	\end{align*}
The left side of \eqref{E:Relation between j-functions} can be written in a similar fashion in terms of hyperbolic functions as
	\[
	j_{\mathfrak{g}}(W)=\det\left(\frac{\id-e^{-\ad_{\mathfrak{g}} W}}{\ad_{\mathfrak{g}} W}\right)
	= \det\left(e^{-\ad_{\mathfrak{g}} W/2}\right)\det\left(\frac{\sinh(\ad_{\mathfrak{g}} W/2)}
		{\ad_{\mathfrak{g}} W/2}\right).	
	\]
Since $\SL(2, \mathbb{R})$ is unimodular, $\det e^{-\ad_{\mathfrak{g}} W/2}= 1$, and, therefore, we get the following neat reformulation of \eqref{E:Relation between j-functions}:
	\begin{equation}\label{E:Relation between j-functions-hyperbolic}
	\det
	\left(
	\frac{\sinh(\ad_{\mathfrak{g}} W/2)}
		{\ad_{\mathfrak{g}} W/2}
	\right)
	=
	\det
	\left(
	\frac{\sinh(\ad_{\mathfrak{m}} W/2)}
		{\ad_{\mathfrak{m}} W/2}
	\right)^2,
	\quad W\in \mathfrak{m}.
	\end{equation}
This condition seems to be worth studying in light of its critical role in the proof of Theorem~\ref{T:simplified-character-sl2} and consequently Theorem~\ref{T:Positivity for SL(2, R)}. The domain of validity of Equation~\eqref{E:Relation between j-functions-hyperbolic} is unknown to this author.
\end{example}
Before we embark on proving positivity of Kirillov's character for the orbit $\mathcal{O}=\Ad^*(\SL(2, \mathbb{R}))H^*$, let us mention that thanks to the existence of a real polarization, as in the case of nilpotent Lie groups, the coadjoint orbit $\mathcal{O}$ exhibits some affine structure, so  we can expect Fourier analysis techniques to be very useful. In particular, the next example shows that the conclusion of Lemma~\ref{L:Orbit Diffeomorphism} is still valid in this case.
\begin{example}\label{Ex:polarization-sl2-orbit} Recall the notation in Example~\ref{Ex:comparison of j-functions}. In the $H^*X^*Y^*$-coordinate system the coadjoint orbit $\mathcal{O}=\Ad^*(\SL(2, \mathbb{R}))H^*$ is a hyperboloid of one sheet. Let $M$ denote the subgroup generated by $\mathfrak{m}$. A short calculation shows that 
	\begin{align*}
	M&=\exp\mathfrak{m}
	=
	\left\{
	\begin{bmatrix}
	a &*\\
	0&a^{-1}
	\end{bmatrix}
	\Biggm |
	a>0
	\right \}.
	\end{align*}
For $a\in \mathbb{R}^+$ and $b\in \mathbb{R}$, let
	\[
	m=
	\begin{bmatrix}
	a & b\\
	0 & a^{-1}
	\end{bmatrix}\in M.
	\]
Then
$m\cdot H^*=H^*-abX^*+abY^*$. Thus, in the $H^*X^*Y^*$-coordinate system, $\Ad^*(M)H^*=\{(1, -c, c)\mid c\in \mathbb{R}\}$, which represents a line in $\mathfrak{sl}(2, \mathbb{R})^*$. See Figure~\ref{F:hyperboloid}. In this way, symplectic geometry can be seen as contributing to our proof of the positivity of Kirillov's character.
\begin{figure}[h]
	\begin{center}
	\includegraphics[height=25mm]{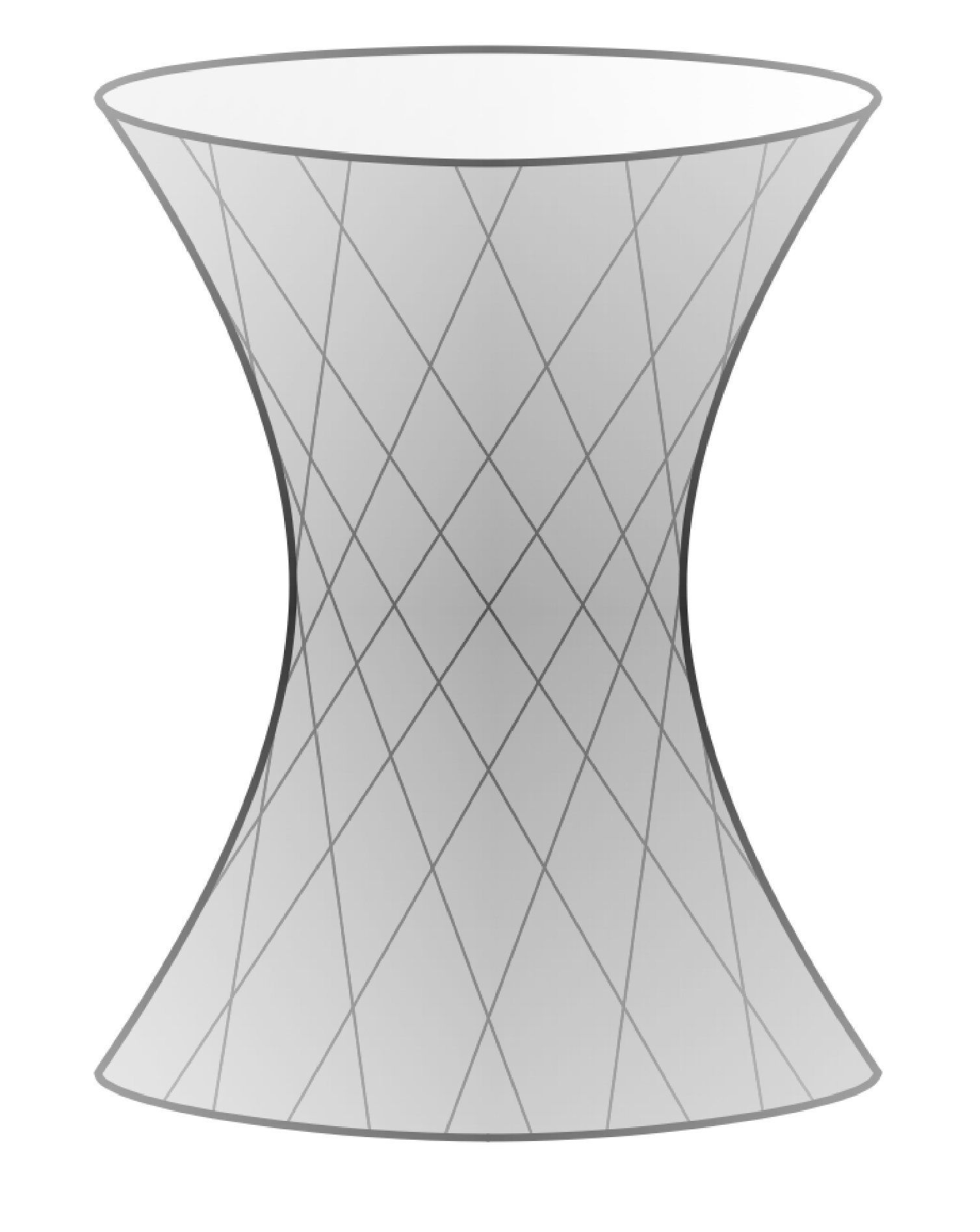}
	\end{center}
	\caption{The coadjoint orbit $\mathcal{O}=\Ad^*(\SL(2, \mathbb{R}))H^*$ as a ruled surface}
	\label{F:hyperboloid}
	\end{figure}
\end{example}
The coadjoint orbit $\mathcal{O}$ above corresponds to a principal series representation of $\SL(2, \mathbb{R})$ which is known to be tempered. So by Rossmann's theorem~\cite{R78}, Kirillov's character $\chi_\mathcal{O}$ is positive. To give a direct proof of this fact, first we establish a generalization of Theorem~\ref{T:Lipsman-Pukanzsky}.
\begin{theorem}\label{T:simplified-character-sl2}
Let $G$ be a unimodular Lie group with Lie algebra $\mathfrak{g}$ and let $\ell\in \mathfrak{g}^*$. Suppose that
	\begin{enumerate}[(i)]
	\item there exists a real polarization $\mathfrak{m}$ subordinate to $\ell$;
	\item the Pukanzsky condition $M/R_\ell\cong (\mathfrak{g}/\mathfrak{m})^*$ of Lemma~\ref{L:Orbit Diffeomorphism} is satisfied; and
	\item $\det
	\left(
	\frac{\sinh(\ad_{\mathfrak{g}} W/2)}
		{\ad_{\mathfrak{g}} W/2}
	\right)
	=
	\det
	\left(
	\frac{\sinh(\ad_{\mathfrak{m}} W/2)}
		{\ad_{\mathfrak{m}} W/2}
	\right)^2,
	\quad W\in \mathfrak{m}$.
	\end{enumerate}
Let $U$ be a sufficiently small neighborhood of $\mathbf{0}\in \mathfrak{g}$ such that the restriction of $\exp\colon \mathfrak{g}\to G$ to $U$ is a diffeomorphism onto $\exp U $. Define $S=\{f\in C^\infty_c(G)\mid  \operatorname{supp}(f)\subset \exp U\}$. Then, with a suitable normalization of measures, for any $\rho$-function $\rho_{(G, M)}$ and any $f\in S$,
	\begin{equation}\label{E:simplified-character-sl2}
	\chi_\mathcal{O}(f)=
	\int_{G/M}\int_{M}
		f(xpx^{-1})e^{i\ell(\log p)}
		\frac{\Delta_M^{-1/2}(p)}{\rho_{(G, M)}(x)}
	\, dp\, dxM
	\end{equation}
where $\chi_\mathcal{O}$ is the Kirillov character for the coadjoint orbit $\mathcal{O}$ through $\ell$.
\end{theorem}
\begin{proof} Assume $f\in S$. By the diffeomorphism $\mathcal{O}\cong G/R_\ell$, which carries one $G$-invariant measure to another, we make a change of variables to obtain
	\begin{align*}
	\chi_{\mathcal{O}}(f)&=
	\int_{\mathcal{O}}\int_{\mathfrak{g}}
		j_{\mathfrak{g}}^{1/2}(X)
		f(\exp X)e^{i\phi(X)}
	\, dX\, d\mu_\mathcal{O}(\phi)\\
	&=
	\int_{G/R_{\ell}}\int_\mathfrak{g}
		j_{\mathfrak{g}}^{1/2}(X)f(\exp X)
		e^{ig\cdot\ell(X)}
	\, dX\, dgR_{\ell}.
	\end{align*}
Recall from Lemma~\ref{L:rho-products} that $\rho_{(G, M)}(xm)\rho_{(M, R_\ell)}(m)=\rho_{(G, M)}(x)$ for $x\in G$ and $m\in M$. Thus, by Corollary~\ref{Co:Quotient integral formula for quasi-invariant measures} and Theorem~\ref{T:QuotientIntegral} we find
	\begin{align*}
	\chi_{\mathcal{O}}(f)&=
	\int_{G/M}\int_{M/R_{\ell}}\int_{\mathfrak{g}}
	j_{\mathfrak{g}}^{1/2}(X)f(\exp X)e^{ixm\cdot\ell(X)}\frac{1}{\rho_{(G, M)}(x)}
	\, dX\, dmR_{\ell}\, dxM\\
	&=
	\left(
	\int_{G/M}\int_{M/R_{\ell}}
	\int_{\mathfrak{g}/\mathfrak{m}}
	\int_{\mathfrak{m}}
	j_{\mathfrak{g}}^{1/2}(P+Y)f(\exp x\cdot(P+Y))e^{im\cdot\ell(P+Y)}
	\right.
	\\
	& \qquad
	\frac{1}{\rho_{(G, M)}(x)}
	\, dP\, dY\, dmR_{\ell}\, dxM
	\left.
	\vphantom{\int_M}
	\right).
	\end{align*}
Here we have made the change of variable $x\mapsto x\cdot X$, and have used the invariance of the Lebesgue measure and $j_\mathfrak{g}$ under the adjoint action of $G$ to simplify.
Define $F_f^x(Y)=\int_{\mathfrak{m}}
	j_{\mathfrak{g}}^{1/2}(P+Y)f(\exp x\cdot(P+Y))e^{i\ell(P+Y)}\frac{1}{\rho_{(G, M)}(x)}
	\, dP$ so that
	\begin{align*}
	\chi_{\mathcal{O}}(f)&=\int_{G/M}\int_{M/R_{\ell}}
	\int_{\mathfrak{g}/\mathfrak{m}}
		F_f^x(Y)
		e^{im\cdot\ell(Y)-i\ell(Y)}
	\, dY\, dmR_{\ell}\, dxM
	\\
	&=
	\int_{G/M}\int_{(\mathfrak{g}/\mathfrak{m})^*}
	\int_{\mathfrak{g}/\mathfrak{m}}
		F_f^x(Y)
		e^{i\gamma(Y)}
	\, dY\, d\gamma
	\, dxM
	\\
	&=
	\int_{G/M}F_f^x(\mathbf{0})\, dxM
	\\
	&=
	\int_{G/M}\int_{\mathfrak{m}}
	j_{\mathfrak{g}}^{1/2}(P)f(\exp (x\cdot P))
	e^{i\ell(P)}\frac{1}{\rho_{(G, M)}(x)}\, dP\, dxM,
	\end{align*}	 
where in the second equality we have used the diffeomorphism $M/R_\ell\cong (\mathfrak{g}/\mathfrak{m})^*$ of Lemma~\ref{L:Orbit Diffeomorphism} whose existence is assumed. The third equality is due to the Fourier inversion theorem. To proceed with the calculations, we observe that since $f\in S$ the change of variables formula applied to the restriction of $\exp$ to $U$ implies
	\[
	\int_G
		\frac{f(g)}{j_{\mathfrak{g}}( \log g)} 
	\, dg
	=
	\int_U
		f(\exp X)
	\, dX.
	\]
Therefore, 
	\begin{align*}
	\chi_\mathcal{O}(f)&=
	\int_{G/M}\int_{\mathfrak{m}}
		j_{\mathfrak{g}}^{1/2}(P)
		f(\exp (x\cdot P))
		e^{i\ell(P)}\frac{1}{\rho_{(G, M)}(x)}
		\, dP\, dxM
	\\
	&
	=\int_{G/M}\int_{M}
		j_{\mathfrak{g}}^{1/2}(\log p)
		\frac{f(xpx^{-1})}{j_{\mathfrak{m}}(\log p)}
		e^{i\ell(\log p)}
		\frac{1}{\rho_{(G, M)}(x)}
	\, dp\, dxM.
	\end{align*}
The third assumption about the relation between the $j$-functions, or equivalently Equation~\ref{E:Relation between j-functions}, allows us to reduce Kirillov's character formula to
	\[
	\chi_\mathcal{O}(f)=
	\int_{G/M}\int_{M}
		f(xpx^{-1})e^{i\ell(\log p)}
		\frac{\Delta_M^{-1/2}(p)}{\rho_{(G, M)}(x)}
	\, dp\, dxM.
	\]
\end{proof}
Finally, we are ready to prove our main positivity result of this section.
\begin{theorem}\label{T:Positivity for SL(2, R)} Let the notation and assumptions be as in Theorem~\ref{T:simplified-character-sl2}. Then 
Kirillov's character $\chi_\mathcal{O}$ is positive on $S$ in the sense that $\chi_\mathcal{O}(f*f^*)\ge 0$ whenever $f*f^*\in S$.
\end{theorem}
\begin{proof}
Let $f*f^*\in S$. Then by \eqref{E:simplified-character-sl2} and writing out the definition of $f*f^*$ we have
	\begin{align*}
	\chi_\mathcal{O}(f*f^*)
	&=
	\int_{G/M}\int_M 
		f*f^*(xpx^{-1})
		e^{i\ell (\log p)}
		\frac{\Delta_M^{-1/2}(p)}{\rho_{(G, M)}(x)}
	\,dp\,dxM
	\\
	&=
	\int_{G/M}
	\left(
	\int_M\int_G 
		f(xpx^{-1}g)f^*(g^{-1})
		e^{i\ell (\log p)}
		\frac{\Delta_M^{-1/2}(p)}{\rho_{(G, M)}(x)}
	\,dg\,dp\,
	\right)
	dxM
	\\
	&=
	\int_{G/M}
	\left(
	\int_M\int_G 	
		f(xpx^{-1}g)
		\overline{f(g)}e^{i\ell (\log p)}
		\frac{\Delta_M^{-1/2}(p)}{\rho_{(G, M)}(x)}
	\,dg\,dp\,
	\right)
	dxM.
	\end{align*}
Applying the change of variable $g\mapsto xg^{-1}$ to the innermost integral and using the unimodularity of $G$ the last integral simplifies to
	\[
	\int_{G/M}
	\left(
	\int_M\int_G 
		f(xpg^{-1})
		\overline{f(xg^{-1})}
		e^{i\ell (\log p)}
		\frac{\Delta_M^{-1/2}(p)}{\rho_{(G, M)}(x)}
	\,dg\,dp\,	
	\right)
	dxM.
	\]
Now, we decompose the measure of $G$ over $M$ and $G/M$ by combining Equation~\ref{E:Functional Equation (G, M)} with the quotient integral formula in Theorem~\ref{T:RhoFunction}.
	\begin{align*}
	\chi_\mathcal{O}(f*f^*)=
	\int_{G/M}
	&\left(
	\int_M\int_{G/M}\int_M 
		f(xpq^{-1}y^{-1})
		\overline{f(xq^{-1}y^{-1})}
		e^{i\ell (\log p)}
	\right.
	\\
	&\quad
	\frac{\Delta_M^{-1/2}(p)}{\rho_{(G, M)}(x)}
	\frac{\Delta^{-1}_M(q)}{\rho_{(G, M)}(y)}	
	\,dq\,dyM\,dp
	\left.
	\vphantom{\int_M}
	\right)
	dxM.
	\end{align*}
Finally, we change the order of the two integrations in the middle and use the formulas in Definition~\ref{D:modular-function} to make the change of variable $p\mapsto p^{-1}q$. 
	\begin{align*}
	&\chi_\mathcal{O}(f*f^*)=
	\left(
	\int_{G/M}\int_{G/M}\int_M\int_M 
		f(xp^{-1}y^{-1})
		\overline{f(xq^{-1}y^{-1})}
		e^{i\ell\left( \log p^{-1}q\right)}
	\right.
	\\
	&\hphantom{\chi_\mathcal{O}(f*f^*)=\int} 	
	\frac{\Delta_M^{-1/2}(p)}{\rho_{(G, M)}(x)}
	\frac{\Delta_M^{-1/2}(q)}{\rho_{(G, M)}(y)}	
	\,dq\,dp\,dyM\,dxM
	\left.
	\vphantom{\int_M}\right)
	\\
	&\hphantom{\chi_\mathcal{O}(f*f^*)} 
	=\left(
	\int_{G/M}\int_{G/M}
	\right.
	\left|
	\int_M 	
	f(xp^{-1}y^{-1})
	e^{-i\ell (\log p)}
	\Delta_M^{-1/2}(p)
	\,dp
	\right|^2 
	\\
	&\hphantom{\chi_\mathcal{O}(f*f^*)=\int}
	\frac{1}{\rho_{(G, M)}(x)}
	\frac{1}{\rho_{(G, M)}(y)}	
	dyM\,dxM
	\left.
	\vphantom{\int_M}\right)
	\ge 0.
	\end{align*}
\end{proof}
\begin{corollary} \hfill \label{C:positivity-of-nilpotents-and-sl2}
\begin{enumerate}[(i)]
\item Let $G$ be a connected (but not necessarily simply connected) nilpotent Lie group. Then for any coadjoint orbit $\mathcal{O}$ of $G$, Kirillov's character $\chi_\mathcal{O}$ is positive on the convolution algebra of Schwartz functions on $G$.
\item Kirillov's character $\chi_\mathcal{O}$ is positive for $\mathcal{O}=\Ad^*(\SL(2, \mathbb{R}))H^*$.
\end{enumerate}
\end{corollary}
\begin{proof} If $G$ is a connected nilpotent Lie group, the first two conditions of Theorem~\ref{T:simplified-character-sl2} are satisfied as stated in Definition~\ref{D:real-polarization} and Lemma~\ref{L:Orbit Diffeomorphism}. The last condition about the relation between the $j$-functions is vacuous since $j\equiv 1$ for any nilpotent Lie group. Finally, observe that in this case we may let $S=C_c^\infty(G)$, or even $S=\mathcal{S}(G)$.

The conditions of Theorems~\ref{T:simplified-character-sl2} and \ref{T:Positivity for SL(2, R)} were verified in Examples~\ref{Ex:polarization-sl2-orbit} and \ref{Ex:comparison of j-functions} for $\mathcal{O}=\Ad^*(\SL(2, \mathbb{R}))H^*$.
\end{proof}
\begin{example} Let $E_{ij}$ be the $3\times 3$ matrix with a one in the $ij$ entry and zeros elsewhere. To simplify the notation, let us write $F$ and $G$ for $E_{11}-E_{22}$ and $E_{11}-E_{33}$, respectively. Then a basis for the eight-dimensional Lie algebra $\mathfrak{g}=\mathfrak{sl}(3, \mathbb{R})$ is given by
	\[
	\mathfrak{B}
	=
	\{
	F, G,
	E_{12}, E_{23}, E_{13}, E_{21}, E_{32}, 
	E_{31}
	\}.
	\]
Consider the polarizing subalgebra
	\[
	\mathfrak{m}=\operatorname{Span}\{F, G, E_{12}, E_{23}, E_{13} \}
	\]
subordinate to $F^*$ consisting of upper triangular matrices in $\mathfrak{sl}(3, \mathbb{R})$. For $W=fF+gG+e_3E_{12}+e_2E_{13}+e_1E_{23}$ in $\mathfrak{m}$, the matrix of $\ad_{\mathfrak{m}} W\colon \mathfrak{m}\to \mathfrak{m}$ 
with respect to the basis 
$\mathfrak{B}'=\{F, G, E_{12}, E_{23}, E_{13}\}$
of $\mathfrak{m}$ is
	\[
	[\ad_{\mathfrak{m}} W]_{\mathfrak{B}'}
	=
	\left[
	\begin{array}{cc|ccc}
	0&0&0&0&0
	\\
	0&0&0&0&0
	\\
	\hline
	-2e_3&-e_3&2f+g&0&0
	\\
	e_1&-e_1&0&g-f&0
	\\
	-e_2&-2e_2&-e_1&e_3&f+2g
	\end{array}
	\right].
	\]
Extend the basis $\mathfrak{B}'$ of $\mathfrak{m}$ to the basis 
$\mathfrak{B}$
of $\mathfrak{g}$. Then the matrix of $\ad_{\mathfrak{g}} W\colon \mathfrak{g}\to \mathfrak{g}$ 
with respect to the basis $\mathfrak{B}$ is
	\[
	[
	\ad_{\mathfrak{g}} W
	]_{\mathfrak{B}}
	=
	\left[
	\begin{array}{c|ccc}
	\multirow{5}{*}{$[\ad_{\mathfrak{m}} W]_{\mathfrak{B}'}$}
	&
	e_3&0&-e_1
	\\
	&
	0&e_2&e_1
	\\
	&
	0&0&e_2
	\\
	&
	-e_2&0&0
	\\
	&
	0&0&0
	\\
	\hline
	\multirow{3}{*}{\Zero}
	&
	-2f-g&0&e_1
	\\
	&
	0&f-g&-e_3
	\\
	&
	0&0&-f-2g
	\end{array}
	\right].
	\]
As in the case of $\SL(2,\mathbb{R})$, this calculation shows that Equation~\eqref{E:Relation between j-functions} holds. So one can obtain the positivity of Kirillov's character for $\mathcal{O}=\Ad^*(\SL(3, \mathbb{R}))F^*$ through Theorem~\ref{T:Positivity for SL(2, R)}.
\end{example}

\section{Construction of Representations for Nilpotent Lie Groups}
In this section we exploit the positivity of Kirillov's character for a connected, simply connected nilpotent Lie group $G$ to construct some representations of $G\times G$. Our method is analogous to the GNS construction in $\mathrm{C}^*$-algebra theory.

Let $G$ be a connected, simply connected nilpotent Lie group. Then for any coadjoint orbit $\mathcal{O}$ of $G$, Kirillov's character $\chi_{\mathcal{O}} \colon C_c^{\infty}(G)\to \mathbb{C}$ is a positive distribution as we proved directly in Corollary~\ref{C:positivity-of-nilpotents-and-sl2}. Moreover, it is straightforward to check that $\chi_{\mathcal{O}}(f_1*f_2^*)=\overline{\chi_\mathcal{O} (f_2*f_1^*)}$. Hence,
\begin{equation}\label{E:Sesquilinear form}
\langle f_1, f_2\rangle_{\chi}=\chi_{\mathcal{O}}(f_1*f_2^*)
\end{equation}
defines a sesquilinear form on $C^{\infty}_c(G)$ that satisfies all the axioms for an inner product except for nondegeneracy, that is, $\langle f, f\rangle_{\chi}=0$ need not imply $f=0$. The remaining axioms are enough to prove the Cauchy--Schwarz inequality
	\begin{equation}\label{E:Cauchy-Schwarz inequality}
	|\langle f_1, f_2\rangle_{\chi}|^2
	\le 
	\langle f_1, f_1\rangle_{\chi}
	\langle f_2, f_2\rangle_{\chi},
	\end{equation}
from which it follows that the set $N=\{f\in C_c^\infty(G)\mid \langle f, f\rangle_{\chi} =0\}$
is a vector subspace of $C_c^\infty(G)$. The formula 
	\begin{equation}\label{E:Inner product on the quotient}
	\langle f_1+N, f_2+N\rangle
	=
	\langle f_1, f_2\rangle_{\chi}
	\end{equation}
defines an inner product on the quotient space $C_c^\infty(G)/N$. We let $\mathcal{H}_\sigma$ denote the Hilbert space completion of $C_c^\infty(G)/N$, and realize $\mathcal{H}_\sigma$ as a representation of $G\times G$ as follows. For $g_1, g_2\in G$ and $f\in C_c^\infty (G)$ define $\lambda_{g_1}\circ \rho_{g_2}(f)$ by
	\[
	\lambda_{g_1}\circ \rho_{g_2}(f)(x)=f(g_1^{-1}xg_2), \quad x\in G.
	\]
Now, we show that this extends to a unitary representation $\sigma$ of $G\times G$ on $\mathcal{H}_\sigma$.
\begin{lemma}\label{L:sesquilinear form preserving}
The left translation $(\lambda, C_c^\infty(G))$ and the right translation $(\rho, C_c^\infty(G))$ preserve the sesquilinear form \eqref{E:Sesquilinear form}. That is, for any $g\in G$,
	\begin{align*}
	 \langle \rho_g f_1, \rho_g f_2\rangle_{\chi}
	&=\langle f_1, f_2\rangle_{\chi},
	\text{ and},\\
	\langle \lambda_g f_1, \lambda_g f_2\rangle_{\chi}
	&=\langle f_1, f_2\rangle_{\chi}.
	\end{align*}
\end{lemma}
\begin{proof}
The left translation invariance of the Haar measure of $G$ gives $f_1*f_2^*(x)=\int_G f_1(y)\overline{f_2(x^{-1}y)}\, dy=\int_G f_1(xy)\overline{f_2(y)}\, dy$. This, combined with the right translation invariance of the Haar measure, implies the right translation invariance of the sesquilinear form. To prove the second assertion, we observe that
	\begin{align*}
	\lambda_g f_1* (\lambda_gf_2)^*(x)
	&=
	\int_G
		f_1(g^{-1}xy)
		\overline{f_2(g^{-1}y)}
	\, dy
	\\
	&=
	\int_G
		f_1(g^{-1}xgy)
		\overline{f_2(y)}
	\, dy,
	\end{align*}
and that the left translation invariance of the sesquilinear form, namely, 
	\[
	\langle \lambda_g f_1, \lambda_g f_2\rangle_{\chi}
	=\langle f_1, f_2\rangle_{\chi}
	\]
follows from the conjugation invariance of the character formula. In more details,
	\begin{align*}
	\chi_\mathcal{O}(c_g\cdot f)&=
	\int_\mathcal{O}\int_\mathfrak{g} f(g^{-1}\exp Xg)e^{i\ell(X)}\, dX\, d\mu_\mathcal{O}(\ell)\\
	&=
	\int_\mathcal{O}\int_G f(g^{-1}xg)e^{i\ell(\log x)}\, dx\, d\mu_\mathcal{O}(\ell)\\
	&=
	\int_\mathcal{O}\int_G f(x)e^{ig^{-1}\cdot\ell(\log x)}\, dx\, d\mu_\mathcal{O}(\ell)\\
	&=
	\int_\mathcal{O}\int_G f(x)e^{i\ell(\log x)}\, dx\, d\mu_\mathcal{O}(\ell)=\chi_\mathcal{O}(f),
	\end{align*}
where we have used the $G$-invariance of the canonical measure $\mu_\mathcal{O}$ in the second to last equality.
\end{proof}
\begin{corollary} The translation maps $(\lambda, C_c^\infty(G))$ and $(\rho, C_c^\infty(G))$ extend to the unitary representations $(L, \mathcal{H}_\sigma)$ and $(R, \mathcal{H}_\sigma)$ of $G$ via
\[
	L_g[f]=[\lambda_g f]
	\text{ and }
	R_g[f]=[\rho_g f].
\]
Therefore, $(\sigma, \mathcal{H}_\sigma)$ defined by $\sigma(g_1, g_2)=L_{g_1}\circ R_{g_2}$ is a unitary representation of $G\times G$.
\end{corollary}
\begin{proof}
The well definedness and unitarity of the operators $L_g$ and $R_g$ are immediate from the Cauchy--Schwarz inequality~\eqref{E:Cauchy-Schwarz inequality} and Lemma~\ref{L:sesquilinear form preserving}.
Moreover, the strong continuity of the left and right regular representations $(\lambda, C_c^{\infty}(G))$ and $(\rho, C_c^{\infty}(G))$ at $g=e_G$ is transferred to $(L, \mathcal{H}_\sigma)$ and $(R, \mathcal{H}_\sigma)$ via the inner product formula~\eqref{E:Inner product on the quotient} and this completes the proof.
\end{proof}

Suppose that the coadjoint orbit $\mathcal{O}$ corresponds, under the Kirillov quantization, to (the class of) the irreducible unitary representation $(\pi, \mathcal{K}_\pi)$ of $G$. A natural question that arises is how the Kirillov $G$-representation associated to $\mathcal{O}$, namely $\mathcal{K}_\pi$, and our $G\times G$-representation $\mathcal{H}_\sigma$---obtained by applying the GNS construction to the positive distribution $\chi_\mathcal{O}$---are related. The answer is given by the next theorem.
\begin{theorem}\label{T:comparison of quantizations} Let $\operatorname{HS}(\mathcal{K}_\pi)$ denote the space of Hilbert--Schmidt operators on $\mathcal{K}_\pi$, and let  $(\eta, \operatorname{HS}(\mathcal{K}_\pi))$ be the representation of $G\times G$ given by $\eta(x, y)(T)=\pi(x)T\pi(y^{-1})$. We have the following $G\times G$-equivariant isometric isomorphisms
	\[
	\mathcal{H}_\sigma\cong \operatorname{HS}(\mathcal{K}_\pi)
	\cong\mathcal{K}_\pi\otimes \mathcal{K}_\pi^*.
	\]
\end{theorem}
\begin{proof} First, recall from Theorem~\ref{T:CharacterFormula} the fact that $\chi_\mathcal{O}(f)=\Tr\pi (f)$. Using this, we have $\langle f, f\rangle_{\chi}=\|\pi(f)\|_{\text{HS}}^2$. Hence, $N=\{f\in C_c^\infty(G)\mid \langle f, f\rangle_{\chi} =0\}=\{f\in C_c^\infty(G)\mid \pi (f)=0\}$. Thus, $[f]\mapsto \pi(f)$ gives a well-defined, injective, linear map $C_c^\infty(G)/N\to \operatorname{HS}(\mathcal{K}_\pi)$. 
Since
	\begin{align*}
	\pi(\sigma(g_1, g_2) f)&=\int_Gf(g^{-1}xg_2)\pi(x)\, dx=\int_Gf(x)\pi(g_1xg_2^{-1})\, dx\\
	&=\pi(g_1)\pi(f)\pi(g_2^{-1})=\eta(g_1, g_2)\pi(f),
	\end{align*}
$[f]\mapsto \pi(f)$ extends to a $G\times G$-equivariant isometric isomorphism of $\mathcal{H}_\sigma$ and $\operatorname{HS}(\mathcal{K}_\pi)$. (For surjectivity, we refer to \cite[Section 18.8]{D77}). This establishes the first isomorphism. The second isomorphism, between $\operatorname{HS}(\mathcal{K}_\pi)$ and $\mathcal{K}_\pi\otimes \mathcal{K}_\pi^*$, is given by the very definition of the tensor product.
\end{proof}
Finally, we make a remark about applying the constructions in this section to non-nilpotent Lie groups. Note that even when the Kirillov character formula is positive, it does not obviously determine a representation. This is because the exponential map is neither injective nor surjective in general, and the test functions supported in a fixed neighborhood of the identity element of the group do not necessarily form an algebra under the convolution operation. This leads to the following question by Higson \cite{Higson13} ``Is there a useful concept of partial representation corresponding to the partially-defined Kirillov character?''
\section{Conclusion} 
By a direct proof, we have shown that Kirillov's character defines a positive trace on the convolution algebra of smooth, compactly supported functions on a connected, simply connected nilpotent Lie group $G$. Then using this and the GNS construction we have produced a unitary representation of $G\times G$ and have established a $G\times G$-equivariant isometric isomorphism between our representation and the Hilbert--Schmidt operators on the underlying representation of $G$. We have proved the positivity of Kirillov's character more generally for coadjoint orbits of unimodular Lie groups under a few additional conditions that are automatically satisfied in the nilpotent case. One hypothesis is the geometric condition required by Lemma~\ref{L:Orbit Diffeomorphism} (which is indeed a statement about Lagrangian fibrations); the other is the existence of real polarizations satisfying Equation~\eqref{E:Relation between j-functions-hyperbolic}.
\section*{Acknowledgement}
The author wishes to thank Nigel Higson for posing the central question in this project and also for numerous conversations about it.

\end{document}